\documentclass[12pt]{article}

\usepackage[english]{babel}
\usepackage[utf8x]{inputenc}
\usepackage{amsmath}
\usepackage{graphicx}
\usepackage{amsfonts}
\newtheorem{theorem}{Theorem}[section]

\newtheorem{corollary}[theorem]{Corollary}

\newenvironment{proof}[1][Proof]{\begin{trivlist}
\item[\hskip \labelsep {\bfseries #1}]}{\end{trivlist}}

\newcommand{\qed}{\nobreak \ifvmode \relax \else
      \ifdim\lastskip<1.5em \hskip-\lastskip
      \hskip1.5em plus0em minus0.5em \fi \nobreak
      \vrule height0.75em width0.5em depth0.25em\fi}

\title{Billiards on Graphs}
\author{Thomas Silverman, Stephen Michael Miller}

\begin{document}
\maketitle

\begin{abstract}
We prove some partial results on the periodicity of billiard systems on graphs. 
\end{abstract}

\section{Introduction}
We consider billiards arranged on the faces of a tetrahedron, one per face, following the edges, such that upon a collision, each billiard goes in the opposite direction obeying conservation of momentum. The cases where each billiard has arbitrary velocity, or arbitrary mass, appear to be too complicated for our methods; we specialise to the case where each billiard has equal mass and velocity. We are able to prove a periodicity result for more than just the tetrahedral case, but rather for any number of billiards travelling along along the proper analogue of the faces on an arbitrary graph (which we first take to have all edge lengths equal, but quickly extend the result to rationally related edge lengths). In particular, this result specializes to the case of $n$ billiards on the unit interval or the circle. Because of our simplifications, for the remainder of this paper, all billiards have the same mass and travel at the same speed.

\section{Periodicity of Billiards on Graphs}

We wish to find an analogue of the faces of a tetrahedron for an arbitrary graph. What we would like is, for each billiard, a path on the graph that it will travel along and an orientation (so that it will be meaningful to say that the billiard reverses direction on collision). The analogue we seek is a directed cycle, defined as follows:\\\\\\
A \emph{directed cycle} of length $n$ in a graph $G$ is a sequence of $n$ vertices $\{v_i\}$ and a sequence of $n$ edges $\{e_i\}$ such that $v_1 = v_{n+1}$ and $e_i$ has endpoints $v_i$ and $v_{i+1}$.\\\\
We will associate to each billiard a directed cycle, so that it will travel from each vertex to the next along the specified edges ($v_i$ to $v_{i+1}$ along $e_i$), and if it collides with another billiard, will travel along the cycle in the reverse direction ($v_{i+1}$ to $v_i$ along $e_i$). The ordering on vertices allows us to distinguish the orientation. However, if some directed cycle follows an edge that connects a vertex to itself, we can't distinguish orientation in this way, so we forbid our graphs to have such edges ("loops"). This in mind, we formalize the billiard system.\\
Let $G$ be a loopless graph. An initial condition for the dynamical system on $G$ is $k$ (not necessarily distinct) directed cycles with a billiard on each directedy cycle. Attached to each billiard is an orientation $\pm 1$ describing which direction of the directed cycle it's traveling on, its location in the directed cycle (i.e., an integer $1 \leq i \leq n$ denoting which edge it's on if the cycle has $n$ edges, and the distance $x\in (0,1]$ to the next vertex). As time passes, billiards travel along the directed cycle at unit speed, only changing course on a collision with another billiard, at which point it reverses orientation and continues on in this direction.\\\\
Note that the system is reversible. Given a state of the system with no two billiards colliding at that point in time, simply reverse the orientation of every billiard on the graph. The time evolution of this state is the reverse time evolution of the original state.\\
\begin{theorem}
Every orbit in $G$ is periodic.
\end{theorem}
\begin{proof}
We consider a modified system. Attached to each billiard in this modified system is its position (as encoded by the edge and distance as before), an orientation, and which directed cycle it's traveling on. Upon collision with a single other billiard, each billiard travels on the other's directed cycle, reversing orientation; effectively, each billiard continues traveling the same direction it was before, along a different cycle. If multiple billiards collide, each billiard stays on its own directed cycle and reverses orientation, traveling the way it came. It suffices to show that this system is periodic, as it's simply the original system with added information - we can recover the original system by permuting the billiards so that each is on its original directed cycle.\\
Now, after unit time, the modified system can take on only a finite number of states: the number of edges in the cycles is finite, as is the number of directed cycles, as is the number of billiards. The only potential obstacle to this is the distance to the next vertex $x$; but since $x$ (the distance to the next vertex) does not change upon collisions in the modified system, and each edge has unit length, the distance to the next vertex is unchanged after unit time. So we have only a finite number of possible states after unit time, so there are integers $m$ and $n$ such that the system is the same after times $m$ and $n$; as the system is reversible, it's periodic, with period dividing $m-n$.
\end{proof}
Note that the proof does not rely on each edge having unit length. If each edge has equal length $\alpha$ then in the proof we simply replace "unit time" by "after time $\alpha$". Similarly, we do not need that all edges have equal length; it suffices that they are rationally related.
\begin{corollary}
If all edge lengths in a graph $G$ are rationally related, then the billiard
graph dynamical system on $G$ is periodic.
\end{corollary}
\begin{proof}
Pick the smallest edge of $G$; say it has length $\alpha$. Then all edges are
of the form $\alpha \frac{p_i}{q_i}$. Set $q:=\textrm{lcm}(q_1,q_2,...q_E)$, with $E$
the cardinality of the edge set. Then subdivide each edge into $\frac{p_iq}{q_i}$
edges of length $\frac{\alpha}{q}$. This produces a graph with each edge of equal
length and a natural bijection between states of this graph and the original.
Replacing unit time with time $\alpha$ in the original proof produces the
corollary.\end{proof}

\begin{corollary}The motion of $n$ balls with the same mass and velocity on the unit circle or interval with elastic collisions is periodic.\end{corollary}
\begin{proof}These are simply special cases of Theorem 2.1. The unit circle is given by the cyclic graph on some number of edges, each ball traversing the same directed cycle (the entire graph), and the unit interval is given by the tree on two vertices, with each directed cycle starting at one vertex, traveling to the next, and returning to where it started.\end{proof} It should be noted that in this case, our proof of Theorem 2.1 is simplified by every billiard following the same directed cycle: swapping directed cycles on collision does not change the direction of movement, but is more akin to assigning a color to each ball and swapping colors on collision. 

\section{Acknowledgements}
The authors would like to thank Misha Guysinsky and Viorel Nitica for their guidance in writing this article, and the Penn State REU program for financial support.
\end{document}